\documentclass[10pt]{amsart}
\usepackage[margin=1.2in,marginparsep=0.1in,marginparwidth=1in]{geometry}

\usepackage{a4, a4wide}
\usepackage{amssymb}
\usepackage{amsmath, bm}
\usepackage{amsthm}
\usepackage{amstext}
\usepackage{amscd}
\usepackage{latexsym}
\usepackage{mathtools}

\usepackage{graphics}
\usepackage{color}
\usepackage{enumitem}
\usepackage[all, cmtip]{xy}
\usepackage{enumitem}
\usepackage[colorlinks, pagebackref]{hyperref}
\hypersetup{
  colorlinks=true,
  citecolor=blue,
  linkcolor=blue,
  urlcolor=blue}

\makeatletter
\def\@tocline#1#2#3#4#5#6#7{\relax
  \ifnum #1>\c@tocdepth 
  \else
    \par \addpenalty\@secpenalty\addvspace{#2}%
    \begingroup \hyphenpenalty\@M
    \@ifempty{#4}{%
      \@tempdima\csname r@tocindent\number#1\endcsname\relax
    }{%
      \@tempdima#4\relax
    }%
    \parindent\z@ \leftskip#3\relax \advance\leftskip\@tempdima\relax
    \rightskip\@pnumwidth plus4em \parfillskip-\@pnumwidth
    #5\leavevmode\hskip-\@tempdima
      \ifcase #1
      \or\or \hskip 2em \or \hskip 2homologyem \else \hskip 3em \fi%
      #6\nobreak\relax
    \dotfill\hbox to\@pnumwidth{\@tocpagenum{#7}}\par
    \nobreak
    \endgroup
  \fi}
\makeatother

\theoremstyle{plain}

\newtheorem{theorem}{Theorem}[section]

\newtheorem{lemma}[theorem]{Lemma}
\newtheorem{corollary}[theorem]{Corollary}
\newtheorem{proposition}[theorem]{Proposition}

\theoremstyle{definition}

\newtheorem{notation}[theorem]{Notation}
\newtheorem{remark}[theorem]{Remark}

\newtheorem{definition}[theorem]{Definition}
\newtheorem{example}[theorem]{Example}

\numberwithin{equation}{section}

\newcommand{\dlim}{\mathop{\varinjlim}\limits}  

\newcommand{\inj}{\hookrightarrow}
\newcommand{\tensor}{\otimes}

\newcommand{\Intersection}{\bigcap}

\newcommand{\Ab}{{\rm Ab}}

\newcommand{\Ker}{{\rm Ker \ }}

\newcommand{\Hom}{{\rm Hom}}

\newcommand{\Spec}{{\rm Spec \,}}

\renewcommand{\tilde}{\widetilde}
\newcommand{\nis}{Nis}

\newcommand{\Z}{{\mathbb Z}}

\newcommand{\KMW}{{K}^{\rm MW}}		

\newcommand{\Pone}{\mathbb P^1}

\newcommand{\A}{\mathbb A}

\newcommand{\w}{\omega}

\newcommand{\calC}{\mathcal C}

\newcommand{\calL}{\mathcal L}
\newcommand{\sL}{\mathcal L}

\newcommand{\calN}{\mathcal N}

\newcommand{\calG}{\mathcal G}

\newcommand{\sO}{\mathcal O}

\newcommand{\sC}{\mathcal C}

\def\<{\langle}
\def\>{\rangle} 
\def\-{\overline} 
\def\~{\widetilde}
\def\^{\widehat}
\def\fr{\mathfrak}
\def\@{\mathcal}
\def\!{\mathscr}
\def\#{\mathbb}
\def\&{\mathbf}
\def\_{\underline}
\def\Dot{\bullet} 
\def\x{\times}
\def\ox{\otimes}

\def\.{\cdot}
\def\Field{\mathcal Field}

\input{xy}
\xyoption{all}

\begin{document}

\title{Milnor-Witt cycle modules over an excellent DVR}

\author{Chetan Balwe}
\address{Department of Mathematical Sciences, Indian Institute of Science Education and Research Mohali, Knowledge City, Sector-81, Mohali 140306, India.}
\email{cbalwe@iisermohali.ac.in}
\author{Amit Hogadi}
\address{Department of Mathematical Sciences, Indian Institute of Science Education and Research Pune, Dr. Homi Bhabha Road, Pashan, Pune 411008, India.}
\email{ amit@iiserpune.ac.in}
\author{Rakesh Pawar}
\address{Department of Mathematical Sciences, Indian Institute of Science Education and Research Pune, Dr. Homi Bhabha Road, Pashan, Pune 411008, India.}
\email{ rakesh.pawar@acads.iiserpune.ac.in}
\date{\today}

\subjclass[2010]{14C17;  14C35; 14F20
(Primary)}
\keywords{Cycle modules; Milnor-Witt K-theory}
\thanks{The second author acknowledges the support of India DST-DFG Project on Motivic Algebraic Topology DST/IBCD/GERMANY/DFG/2021/1.}

\begin{abstract} The definition of Milnor-Witt cycle modules in \cite{Feld} can easily be adapted over general regular base schemes. However, there are simple examples (see (\ref{eg1})) to show that Gersten complex fails to be exact for cycle modules in general if the base is not a field. The goal of this article is to show that, for a restricted class of Milnor-Witt cycle modules over an excellent DVR satisfying an extra axiom, called here as R5, the expected properties of exactness of Gersten complex and $\A^1$-invariance hold. Moreover R5 is vacuously satisfied when the base is a field and it is also satisfied by $\KMW$ over any base. As a corollary, we obtain the strict $\A^1$-invariance and the exactness of Gersten complex for $\KMW$ over an excellent DVR. 
\end{abstract}

\maketitle

\section{Introduction}

Rost's theory of cycle modules was generalized to Milnor-Witt cycle modules (or MW-cycle modules) by N. Feld in \cite{Feld}. With minor modifications and verifications, this definition  makes sense as it is, over a regular excellent base scheme. 

Over a field, an MW-cycle module $M$ satisfies two additional properties (see \cite[8.1, 9.1]{Feld}) implied by the axioms, viz. local acylicity and homotopy invariance. The definition of a cycle module makes it possible to write a Gersten complex for it over any scheme equipped with a graded line bundle. Local acyclicity is a consequence of exactness of this Gersten complex in positive degrees.
For a MW-cycle module over a positive dimensional base scheme, the Gersten complex may not be exact (see Example \ref{eg1}) in positive degrees. We introduce an additional axiom, called R5, to remedy this situation. The axiom R5 says (see Definition \ref{R5}) that for a field $F$ and finitely many horizontal valuations (see Definition \ref{horizontal}) $v_1,..,v_n$, the sum of residue morphisms 
$$ M(F, *) \xrightarrow{\sum_{i=1}^n \partial_{v_i}} \bigoplus_{i=1}^n M(\kappa(v_i),*)$$ 
is surjective. 

Given a cycle module $M$ and a graded line bundle $\calL$ on a scheme $X$, one has a naturally defined Nisnevich sheaf, $M(X, \calL)$ (see \eqref{cmsheaf}). The following is one of the main results of this paper:

\begin{theorem}[Local acyclicity]\label{main}Let $S$ be the spectrum of an excellent DVR. Let $M$ be a MW-cycle module on $S$ satisfying the property R5 (see Definition~\ref{R5}).
Let $X$ be an essentially smooth semi-local $S$-scheme and  $\calL$ be a graded line bundle on $X$,  then the Gersten complex $C^{\bullet}(X,M,\sL)$ is exact in degrees $\geq 1$. 
In particular, $$H^i_{Nis}(X,M(X,\sL)) = 0 \ \forall \ i>0.$$ 
\end{theorem}

\noindent  The requirement for this additional axiom R5 can be justified as follows. Whenever the field $F$ appearing in the condition R5 is the function field of a smooth $S$-scheme, R5 is exactly equivalent to the exactness of Gersten complex for a suitable semi-local one dimensional scheme. In particular, R5 is motivated from a property always satisfied by MW-cycle modules over a perfect field (see Proposition \ref{r5field} for an elementary argument). Theorem \ref{main} may be loosely paraphrased by saying that exactness of Gersten complex for general semi-local essentially smooth schemes over a DVR follows from local acyclicity in the  one-dimensional case. However, this should be taken with the caveat that R5 only refers to horizontal valuations and it is applied to fields $F$ which are not necessarily function fields of smooth schemes. The strategy of the proof is motivated by \cite[Theorem 3.1]{lueders20}, which is similar to the original argument in \cite[\S 6]{Rost}, and is also used in \cite[8.1]{Feld}. 

\begin{remark}
Note that the vanishing in Theorem \ref{main} implies that the Zariski and the Nisnevich cohomology groups of $M(X,\sL)$ agree. 
\end{remark}

In the last section, we show (see Theorem \ref{a1inv}) that strict $\A^1$-invariance holds for MW-cycle modules satisfying R5. 
\begin{theorem}\label{a1inv} Let $X$ be a smooth $S$-scheme. Then for all $i\geq 0$,
$$H^i_{Nis}(X, M(X,\sL))\simeq  H^i_{Nis}(\A^1_X,M(\A^1_X, \sL \tensor (\w_{\A^1_X/X}^{\vee}, -1))).$$ 
\end{theorem}

\noindent The above theorem is deduced from local-acyclicity and $\A^1$-invariance in the field case using a spectral sequence argument.  The main example of a MW-cycle module is $\KMW$ itself. 
 \begin{theorem}\label{kmwcheck} Let $S$ be a regular excellent scheme. Then 
 $\KMW$ is a MW-cycle module over $S$ satisfying the property R5.
 \end{theorem}
As a corollary of the above results, we get\begin{corollary}\label{kmwa1}
Over an excellent DVR, the Milnor-Witt sheaves $K_n^{MW}$ are strictly $\A^1$-invariant. 
\end{corollary}

\begin{remark}
The entire argument of \eqref{kmwa1} and also analogues of Theorems \ref{main} and \ref{a1inv} also work with MW-cycle modules replaced by cycle modules as defined by Rost using Milnor K-theory. In particular the arguments in this paper also show that the Milnor K-theory sheaves $K^M_n$, are strictly $\A^1$-invariant. 
When $S$ is the spectrum of a field $k$, (not necessarily perfect), the strict $\A^1$-invariance for $K_n^M$ is proved in~\cite[Proposition 8.6]{Rost}. When $k$ is a perfect field, strict $\A^1$-invariance of $\KMW_n$ was shown in~\cite[Theorem 5.38]{morel}.  
\end{remark} 

\noindent {\bf Acknowledgement}: We thank Anand Sawant for useful comments and suggestions. The third author acknowledges the support of NBHM Postdoctoral Fellowship and IISER Pune.

\section{Preliminaries on cycle modules over a base}

Let $S$ be a regular excellent scheme (see \cite[\href{https://stacks.math.columbia.edu/tag/07P7}{Tag 07QT}]{stacksproj}). Let $\Field/S$ be the category of finitely generated $S$-fields. The assumption on $S$, allows us to define, for every $F/S \in \Field/S$, the canonical invertible sheaf $\omega_{F/S}$ following \cite[6, Definition 4.3]{liu} (see Definition \ref{canonicalsheaf}). We first recall the definition of  Milnor-Witt cycle modules over $S$, exactly as in \cite[Section 3]{Feld}. The only minor difference is that we use graded line bundles instead of virtual vector bundles for coefficients. This results in minor straightforward changes, for e.g. in the definition of pull back via smooth map (see Definition \ref{maps}(2)) we use canonical sheaf instead of sheaf of differentials.  By \cite[Lemma 3.1.3]{feld21}, using graded line bundles does not amount to loss of generality while at the same time this enables us to talk about cycle modules over imperfect fields.

\begin{definition}\label{horizontal}
For an $S$-field $F$, a valuation $v$ on $F$ is called an $S$-valuation 
if ${\rm Im}( \sO(S)\to F) \subset \sO_v$. One can show that since $S$ is universally Japanese, this condition is equivalent to $\sO_v$ being a localization of a finite type normal $S$-scheme. We say that the valuation $v$ is \textit{horizontal} if the images of $\Spec(F)$ and $\Spec(\kappa(v))$ in $S$ are distinct. 
\end{definition}

We first consider the category $\fr F_S$, defined as follows:
\begin{enumerate}
\item The objects of $\fr F_S$ are pairs $(F, \calL_F)$ where $F$ is a finitely generated field over $S$ and $\calL_F=(L, n)$ is a graded line bundle on $F$, i.e. a pair consisting of a one-dimensional $F$-vector space $L$ and an integer $n$ (see Appendix). 
\item A morphism $\phi: (E, \calL_E) \to (F, \calL_F)$ is a pair consisting of a homomorphism $E \to F$ and an isomorphism $\calL_F \cong \calL_E \otimes_E F$. Such a morphism is said to be a \emph{finite extension} if the homomorphism $E \to F$ is a finite extension. 
\end{enumerate}
We denote, as in  \cite{Feld}:
$$ \KMW(E,  \calL_E):= \KMW_n(E)\tensor_{\Z[E^{\x}]} {\Z[L^{\x}]} $$ where $\calL_E = (L,n).$
To recall the definition of Milnor-Witt cycle modules following  \cite[1.1]{Deglise} and \cite[3.1]{Feld} we first construct the category  $\widetilde{\fr F}_S$, having the same objects as $\fr F_S$, but the morphisms of which are defined via the following generators and relations. We freely use notation from \cite[2.2]{Feld}.

\noindent {\bf Generators}: 
\begin{itemize}
\item[{\bf D1}] $\phi_*: (E, \calL_E)\to (F, \calL_F)$ for any $\phi: (E, \calL_E) \to (F, \calL_F)$ in $\fr F_S$. 
\item[{\bf D2}] $\phi^*: (F, \w_{F/S}\otimes\calL_F)\to (E, \w_{E/S}\otimes\calL_E)$ for any finite extension $\phi: (E, \calL_F) \to (F, \calL_F)$ in $\fr F_S$.
\item[{\bf D3}] $\gamma_x: (E, \calL_E)\to (E, \calL'_E\otimes\calL_E)$ for $x\in  \KMW(E,  \calL'_E)$.
\item[{\bf D4}] $\partial_v: (F, \calL_F)\to (\kappa(v), \calN_v^{\vee}\otimes \calL_{\kappa(v)})$ where 
	\begin{enumerate}
	\item[(i)]  $F\in \Field/S$ and $v$ is a discrete $S$-valuation on $F$ over $S$.
	\item[(ii)] $\calL$ is a graded line bundle on $\sO_v$, such that $\calL_F=\calL\tensor_{\sO_v}F$ and $\calL_{\kappa(v)} = \kappa(v)\tensor_{\sO_v}\calL$.
	\item[(iii)] $\calN_v:= (\fr m_v/\fr m_v^2,1)$.
	\end{enumerate}	
\end{itemize}

\noindent {\bf Relations}:

\begin{enumerate}
\item[\bf{R0}] For $x\in  \KMW(E,  \calL_E)$ and $y\in  \KMW(E,  \calL'_E)$, $\gamma_x\circ\gamma_y=\gamma_{x\.y}$.
\item[\bf{R1a}] $(\psi\circ\phi)_*=\psi_*\circ\phi_*$ for morphisms $\phi$ and $\psi$ in $\fr F_S$. 
\item[\bf{R1b}] $(\psi\circ\phi)^*=\phi^*\circ\psi^*$ for morphisms $\phi$ and $\psi$ in $\fr F_S$. 
\item[\bf{R1c}] Let $\phi: (F, \calL_F)\to (E, \calL_E)$ and $\psi: (E, \calL_E)\to (L, \calL_L)$ be in $\fr F_S$ such that with $E \to F$ is finite and $E\to L$ is separable. Let $R=E\otimes_L F$. For each $p\in \Spec R$, let $\phi_p: (L, \calL_L)\to (R/p, \calL_{R/p})$ and $\psi_p: (F, \calL_F)\to (R/p, \calL_{R/p})$ be morphisms in $\fr F_S$ induced by $\phi$ and $\psi$. Then
$$ \psi_*\circ\phi^*=\sum_{p\in \Spec R} \phi_p^*\circ (\psi_p)_*.$$

\item[\bf{R2}] Let $\phi: (E, \calL_E)\to (F, \calL_F)$ be in $\fr F_S$, $x\in  \KMW(F, \w_{F/S}\otimes \calL_F)$ and $y\in  \KMW(E,  \calL'_E)$. Then the following relations hold:
\begin{enumerate}
\item[\bf{R2a}] $\phi_*\circ \gamma_x=\gamma_{\phi_*(x)}\circ \phi_*$.
\item[\bf{R2b}] For $\phi$ finite,  $\phi^*\circ \gamma_{\phi_*(x)}=\gamma_x\circ \phi^*$.
\item[\bf{R2c}] For $\phi$ finite,  $\phi^*\circ \gamma_y\circ \phi_*=\gamma_ {\phi^*(y)}$.
\end{enumerate}
\item[\bf{R3a}] Let $\phi: E\to F$ be a field extension and $w$ be a valuation on $F$ which restricts to a non trivial valuation $v$ on $E$ with ramification index $e$. Let $\calL$ be a graded line bundle on $\sO_v$, so that we have a morphism $\phi : (E, \calL_E )\to (F, \calL_F)$ which induces a morphism $\-{\phi}: (\kappa(v), -\calN_v \otimes \calL_{\kappa(v)}) \to (\kappa(w), -\calN_w\otimes\calL_{\kappa(w)})$. Then
$$ \partial_w\circ \phi_*=\gamma_{e_{\epsilon}} \circ \-{\phi}_*\circ \partial_v.$$
\item[\bf{R3b}] Let $E\to F$ be a finite extension of fields, let $v$ be a valuation on $E$ and let $\calL$ be a graded line bundle on $\sO_v$. For each extension $w$ of $v$, we denote by $\phi_w : (\kappa(v),\calL_{\kappa(v)})\to (\kappa(w), \calL_{\kappa(w)})$ the morphism induced by $\phi: (E, \calL_E )\to (F, \calL_F)$. We have
$$ \partial_v\circ \phi^*=\sum_w (\phi_w)^*\circ \partial_w.$$
\item[\bf{R3c}] Let $\phi: (E, \calL_E)\to (F, \calL_F)$ be a morphism in $\fr F_S$. Let $w$ be a valuation on $F$ that is trivial on $E$. Then $$\partial_w\circ \phi_*=0.$$
\item[\bf{R3d}] Let $\phi$ and $w$ be as in R3c, and let $\-\phi : (E, \calL_E)\to (\kappa(w),\calL_{\kappa(w)})$ be the induced morphism. For any uniformizer $\pi$ of $v$, we have
$$\partial_w\circ \gamma_{[-\pi]}\circ \phi_*=\-{\phi}_*.$$
\item[\bf{R3e}] Let $E$ be a field over $S$, $v$ be a valuation on $E$ and $u\in\sO_v^{\x}.$ Then
$$\partial_v\circ\gamma_{[u]}=\gamma_{\epsilon[\-u]}\circ\partial_v$$ and
$$\partial_v\circ\gamma_{\eta}=\gamma_{\eta}\circ \partial_v.$$
\item[\bf{R4a}] Let $(E, \calL_E)$ and let $\Theta$ be an endomorphism of $(E, \calL_E)$ given by an automorphism of $\calL_E$. Denote
by $\Delta$ the canonical map from the group of automorphisms of $\calL_E$ to the group $\KMW(E,0)$. Then $\Theta_*= \gamma_{\Delta(\Theta)} : (E, \calL_E)\to (E, \calL_E)$.
\end{enumerate}

\begin{definition}
A MW-pre-cycle module $M$ over $S$ is a covariant functor $M: \~{\fr F}_S\to \Ab$, to the category $\Ab$ of abelian groups. A morphism of pre-cycle modules is a natural transformation between the corresponding functors. We will denote the maps $\partial_v, \gamma_x, \phi^*, \phi_*$ etc. for $M(\partial_v), M(\gamma_x), M(\phi^*), M(\phi_*)$ respectively throughout the paper to ease the notation. 
\end{definition}

\begin{definition} Let $M$ be a pre-cycle module over $S$,  $X$ a scheme over $S$ and $\calL$ a graded line bundle $X$.
 \begin{enumerate}
 \item For a point $x\in X$, $M(x, \calL_X)=M(\kappa(x), \Lambda_{x}\otimes\calL_X)$ 
 where $\Lambda_x:=(\w_{x/S}, \dim \-{\{x\}})$ and $\omega_{x/S}$ is the canonical bundle (see \ref{canonicalsheaf}).
 \item Suppose $X$ is normal,  $\xi_X$ is the generic point of $X$ and $z$ is a codimension 1 point of $X$, then $$\partial^{\xi_X}_z:=\partial_{v_z}: M(\xi_X, \calL_X)\to M(\kappa(z), \calL_X).$$
 \item Let $x\in X$ and let $Z$ be the normalization of the reduced closed subscheme $\-{\{x\}}_{red}$.  Let $\{z_1, \cdots, z_r\}\subset Z^{(1)}$, where each $z_i\mapsto y$ under the map $Z\to X$ and $$\phi_{z_i}: (\kappa(y), \calL_X) \to (\kappa(z_i), \calL_X).$$ We define $\partial^x_y: M(x, \calL_X)\to M(y, \calL_X) $ as follows 
 \begin{align}\partial^x_y:=\begin{cases} \sum_{z_i}\phi_{z_i}^*\circ \partial^x_{z_i} & \text{if $y\in\-{\{x\}}$,}\\
      0            & \text{if $y\notin\-{\{x\}}$.}
 \end{cases}
 \end{align}

 \end{enumerate}

\end{definition}
\begin{definition}\label{CM} \cite[Definition 4.2]{Feld}
A MW-pre-cycle module $M$ over $S$ is a MW-cycle module if it satisfies the following two properties:
\begin{enumerate}
\item[(FD)] Finite support of divisors: Let $X$ be a normal scheme, $\calL_X$ be a graded line bundle over $X$ and $\rho$ be an element of $M(\xi_X, \calL_X)$. Then $\partial_x(\rho) = 0$ for all but finitely many $x \in X^{(1)}$.

\item[(C)] Closedness: Let $X$ be an integral $S$-scheme which is local of dimension 2 and $\calL_X$ be a graded line bundle over $X$. Then $$0=d\circ d=\sum_{}\partial^x_{x_0}\circ\partial^{\xi}_x: M(\xi,\calL_X)\to M(x_0, \calL_X)$$
\end{enumerate}
where $\xi$ is the generic point and $x_0$ is the closed point of $X$.
\end{definition}

\begin{definition}\label{cyclecomplex} \cite[7.1]{Feld} Let $M$ be a 
 MW-cycle module $M$ on $S$. Let $X/S$ be any scheme which is equi-dimensional of dimension $d$. Let $\sL$ be a graded line bundle on $X$. One defines a Gersten complex $C^{\Dot}(X, M, \calL)$ by 
$$C^p(X, M, \calL)=\bigoplus_{x\in X^{(p)}}M(x, \calL)$$
where $X^{(p)}$ denotes the set of points of codimension $p$ in $X$. As in \cite[7.1]{Feld}, when $X$ is not equi-dimensional one can also define a homological version of the Gersten complex $C_{\Dot}(X, M, \calL)$ indexed by the dimension of the points instead of the codimension. However, throughout this paper we only use the cohomological indexing. We let 
$A^p(X,M, \sL)$ denote the $p$-th cohomology of $C^{\Dot}(X, M, \calL)$.

Let $Z$ be a closed subset of $X$ and $U:=X-Z\xrightarrow{j}X$, the open immersion of $U$ in $X$, then the Gersten complex for $M$ on $X$  with support on $Z$ is given by  

$$C^{\Dot}_Z(X, M, \calL):=\Ker\bigg(C^{\Dot}(X, M, \calL)\xrightarrow{j^*} C^{\Dot}(U, M, \calL)\bigg).$$

$A^p_Z(X,M, \sL)$ denotes the $p$-th cohomology of the complex $C^{\Dot}_Z(X, M, \calL).$
\end{definition}

\begin{definition}\label{cmsheaf}
Given a MW-cycle module $M$ over $S$ and a graded line bundle $\calL$ on a smooth $S$-scheme $X$, we have a naturally defined Zariski sheaf (see \cite[8.5]{Feld}) $M(X,\calL)$ on $X$ given by 
$$ U\mapsto A^0(U,M,\sL).$$
It is left to the reader to check that R3a implies that this is in fact a Nisnevich sheaf on $X$. \end{definition}

\begin{remark}
Given MW-cycle module $M$ on $S$ and a morphism $f:Y\to X$ of smooth $S$-schemes one shows that for a graded line bundle $\sL$ on $S$, we have using R3a a well defined map 
$ A^0(X,M,\sL_X) \to A^0(Y,M,\sL_Y)$. As in \cite[Theorem 5.3.1]{feld21}, one can easily show that $A^0(-,M,\sL)$ is a birational invariant for smooth proper $S$-schemes. 
\end{remark}

\begin{proposition}\label{hwr}
Let $M$ be a MW-cycle module on $S$ and $\sL$ a graded line bundle on $S$. Then $M$ satisfies the following two properties:
\begin{enumerate}
\item[(WR)]Weak reciprocity: For every $S$-field $F$, let $\partial_{\infty}$ be the valuation on $F(u)$ at infinity, then 
$$\partial_{\infty}(A^0(\A^1_F, M, \sL))=0.$$ 
\item[(H)] For every $S$-field $F$, the complex
$$0\to M(F, \sL)\xrightarrow{r} M(\xi, \sL)\xrightarrow{d} \bigoplus_{x\in (\A^1_F)^{(1)}} M(x, \sL)\to 0$$  is exact. Here $\xi$ denotes the generic point of $\A^1_F$.
\end{enumerate}
\end{proposition}
\begin{proof}
The proof is as in~\cite[Proposition 2.2, Theorem 2.3]{Rost}, with necessary modifications as in the identities ${ab}=\{a\}+\{b\}$ replaced by $[ab]=[a]+[b]+\eta[a][b]$, etc.
\end{proof}

\begin{example}\label{eg1}
Let $S=\Spec(R)$ where $R$ is an excellent DVR. Let $\eta$ and $s$ denote its generic and closed points respectively. Let 
\begin{align*}
M(F, \calL_F):= \begin{cases}
\KMW(F, \calL_F) & \text{if $F$ lies over } s,\\
0         & \text{otherwise}.
\end{cases} 
\end{align*} 
It is easy to see that $M$ has the structure of a cycle module, using the cycle module structure on $\KMW$ on the $\kappa(s)$. The Gersten complex of $M$  on $S$ is
 $$0\to M(\eta,\calL_S)\xrightarrow{\partial} M(s, \calL_S)\to 0$$
 which is clearly not exact in degree 1. 
 \end{example}

First, we make some conventions applicable throughout the paper, unless mentioned otherwise.
\begin{notation}
For a given Milnor-Witt pre-cycle module $M$ over $S$, we simply write $M(E,*)$ when the coefficient is implicitly determined by the context or when it can be any arbitrary coefficient. For example, given a field $E$ and an $S$-discrete valuation $v$ instead of writing 
$$ M(E,\calL_E) \xrightarrow{\partial_v} M(\kappa(v),\calN^{\vee}_v\tensor \calL_{\kappa(v)})$$
we sometimes simply write 
$$ M(E,\calL_E) \xrightarrow{\partial_v} M(\kappa(v),*)$$ 
The axioms canonically tell us what coefficient to fill in place of $*$ above. Moreover, if the coefficient $\calL_E$ is not of any particular interest and can be arbitrary, we simply write the above map as 
$$ M(E, *) \xrightarrow{\partial_v} M(\kappa(v),*)$$
This abuse of notation, which will only be required in this section, is to facilitate the focus on the idea, instead of coefficients, whenever they are anyway canonically determined or are arbitrary. 

\end{notation}
In order to find a good class of Milnor-Witt cycle modules over a DVR which satisfy local acylicity, we first observe the following. 
\begin{proposition}\label{r5field}
Let $M$ be a Milnor-Witt cycle module over a perfect field $k$. 
Let $F \in \Field/k$ with distinct discrete $k$-valuations $v_1,..,v_n$. 
Let $\calL$ be a graded line bundle on the semi-local ring $ \sO = \Intersection_i \sO_{v_i}$.Then
 $$\ M(F, *)\xrightarrow{\sum_{i=1}^n\partial_{v_i}} \bigoplus_{i=1}^n M(\kappa(v_i),*)$$ is surjective.
\end{proposition}
\begin{proof}
The statement of this proposition simply says that local acyclicity holds for the semi-local ring $\sO$. Thus, the claim for essentially smooth semilocal ring $\sO$, directly follows from Feld (see \cite[Theorem 8.1]{Feld}). However, we give a direct proof which is shorter. To show this, it is enough to show that for any $1\leq j\leq n$, and an $\alpha\in M(\kappa(v_j),*)$, there exists a $\beta \in M(F, *)$ such that 
$$\partial_{v_{j}}(\beta)=\alpha \ \ \  \text{and}  \ \ \ \partial_{v_{i}}(\beta)=0 \ \forall\  i\neq j. \ \ \ ({\star})$$
Without loss of generality, we may assume $j=1$. 

\noindent \underline{Step 1}: Choose a finite extension $E/F$ such that there exists a valuation $w_1$ of $E$ extending the valuation $v_1$ on $F$ satisfying  
\begin{enumerate}
\item $k(w_1) / k(v_1)$ is a trivial extension.
\item There exists a field $l\subset E$ which maps isomorphically onto $k(w_1)$.
 \end{enumerate}
 To see this, let $F^h$ be the fraction field of the Henselization $\sO^h_{v_1}=\^{\sO}_{v_1}\cap F^{sep}$ of $\sO_{v_1}$ which has the residue field $k(v_1)$. Existence of a finite subextension $E/F$ of $F^h/F$ follows from the fact that $k(v_1)/k$ is a separably finitely generated field extension. 

Let $\Sigma$ be the set of all valuations of $E$ except $w_1$, which lie over one of the $v_i$ for $1\leq i\leq n$. Choose a uniformizer $\pi$ of $w_1$ such that $\pi\in F$ and $\pi \equiv -1 \ {\rm mod} \ {\mathfrak m}_{w}$ for all $w\in \Sigma$. \\

\noindent \underline{Step 2}: First observe that (using the fact that $E$ contains the residue field $k(w_1)$)  
 $$\partial_{w_1}\circ\gamma_{[-\pi]}: M(E, *)\to M(k(w_1), *)$$
is surjective. We leave it to the reader to check that this follows directly from R3d. Thus, there exists a $\tilde{\beta} \in M(E,\calL_E)$ such that its image via the above map followed by the isomorphism 
$$ M(k(w_1), *) \to M(k(v_1), *)$$ is $\alpha$. 
Let $\beta = \phi^*\tilde{\beta}$. where $\phi$ is the map from $(F, \calL_F)\to (E, \calL_E)$ We will need to show that $\beta$ satisfies the required property mentioned in equation $(\star)$.  
This follows from R3b and Lemma \ref{pi=1}. Details are left to the reader.
\end{proof}

\begin{lemma}\label{pi=1}
Let $M$ be a cycle module over a base $S$. Let $F$ be an $S$-field and $v$ be a $S$-valuation on $F$. Let $c\in F$ such that $c\equiv 1 \ {\rm mod} \ {\mathfrak m}_v$. Then $\partial_v \circ \gamma_{c} = 0$. 
\end{lemma}
\begin{proof}
Since $\epsilon[\overline{c}] = 0$, this is a direct consequence of R3e. 
\end{proof}

To avoid the non-acyclic example~\ref{eg1} above, we consider the following class of Milnor-Witt cycle modules.
\begin{definition}[{\bf R5}]\label{R5}
Let $M$ be a MW-cycle module over $S$. $M$ is said to satisfy property R5 if
for every field $F \in \Field/S$ and horizontal $S$-discrete valuations $v_1,..,v_n$, let $\calL$ be a graded line bundle on the semi-local ring $ \sO = \Intersection \sO_{v_i}$. Then
 $$\ M(F, \calL_F)\xrightarrow{\sum_{i=1}^n\partial_{v_i}} \bigoplus_{i=1}^n M(\kappa(v_i),*)$$ is surjective.
 \end{definition}

We now prove Theorem \ref{kmwcheck} to show that $\KMW$ is a cycle module satisfying R5. 
\begin{proof}[Proof of Theorem \ref{kmwcheck}] We leave it to the reader to check that $\KMW$ is MW-pre-cycle module as the proofs are exactly the same as in the field case. It thus remains to check closedness property (C) in the Definition~\ref{CM} and property R5. \\

\noindent \underline{Step 1}: In this step we show that $\KMW$ satisfies R5. To see this, let $E$ be an $S$-field and $v_1,v_2,..,v_n$ be finitely many distinct valuations. As in the proof of Theorem \ref{r5field} it is enough to show that for any $\alpha \in \KMW(k(v_1),*)$ there exists a $\beta$ such that 
$\partial_{v_1}(\beta)=\alpha$ and $\partial_{v_i}(\beta)=0$ for all $i>1$. 

\noindent We first claim that for any uniformizer $\pi$ of $v_1$, the specialization map 
$$s^{\pi}_v=\partial_v\circ\gamma_{[-\pi]}: \KMW(F, \calL_F)\to \KMW(\kappa(v), *)$$
is surjective. This follows from the fact that if  $\alpha \in \KMW_n(k(v_1), (({\mathfrak m}_1/{\mathfrak m}_1^2)^{\vee}, -1))$ is of the form $[u_1,\cdots,u_n]\tensor \overline{\pi}^{\vee}$ then for lifts $\tilde{u}_i$ of $u$ in $\sO_{v_1}$, we have 
$$\partial_{v_1}([-\pi][\tilde{u}_1, \cdots ,\tilde{u}_n]) = \alpha$$
Moreover if we choose $\pi \equiv -1 \ {\rm mod} \ {\mathfrak m}_j$ for $j>1$, then clearly 
$$\partial_{v_j}([-\pi][\tilde{u}_1, \cdots ,\tilde{u}_n]) = 0 \ \forall \ j>1$$ thus showing R5.\\

\noindent \underline{Step 2}: In the remaining steps we verify the closedness property \eqref{CM}(C) following analogous proofs for Milnor K-theory in~\cite[Proposition 1]{kato86} and~\cite[Proposition 49.30]{Elm-Kar-Mer}. As in ~\cite[Proposition 1]{kato86}, it suffices to check \eqref{CM}(C) when $X=\Spec(A)$ where $(A,\fr m)$ is a complete local domain of dimension 2. By Cohen's structure theorem (see ~\cite[\href{https://stacks.math.columbia.edu/tag/032D}{Tag 032D}]{stacksproj}), there is a complete DVR $\Lambda\subset A$ such that $A$ is finite over $\Lambda[[t]]$ for an indeterminate $t$. Let $\pi$ be the uniformizer of $\Lambda$. Let $A/\fr m=\kappa$. Let $K={\rm Frac}(A).$
It remains to prove the claim that 
$$\KMW_{n+2}(K)\xrightarrow{\sum_{x\in X_{(1)}} d_x} \bigoplus_{x\in X_{(1)}} \KMW_{n+1}(\kappa(x), \Lambda_x)\xrightarrow{\sum_{x\in X_{(1)}} d'_x} \KMW_{n}(\kappa, (\det \fr m/\fr m^2)^{\vee}) $$
is a complex.\\

\noindent \underline{Step 3}: We will first prove the claim in Step 2 when $m:=n+2\geq 1$. 
In this case, $\KMW_{m}(K)$ is generated by $[u_1]\cdots [u_{m}]$ for $u_i\in K^{\x}.$ 
Moreover, $K^\x$ is generated by $A^\x, H, \pi, t$ where $H$ is the subgroup of $K^\x$ generated by elements of the form 
$$1+a_1t^{-1}+\cdots+ a_rt^{-r} \ \text{where} \ a_i  \in (\pi) \ \forall \ i \ \text{and} \  r \geq 0.$$
Let $H'$ be the subgroup generated by $H, \pi, t$. Using this fact, and the results in \cite[3.17 (2), (3), 3.5 (1)]{morel}, we leave it for the reader to verify that proving the claim is equivalent to showing the following 

 \begin{equation}\label{(2)}
\sum_{x\in X_{(1)}} d'_x\circ d_x([\pi][t])=0
\end{equation}
\begin{equation}\label{(3)}
\sum_{x\in X_{(1)}} d'_x\circ d_x([f][g_1]\cdots[g_{n+1}])=0
 \end{equation}
for all $f\in H$ and $g_i\in H'$.

%

 By Weierstrass preparation theorem, the height 1 prime ideals in $A$ are given by either $(\pi)$ or $(f)$ where $f\in \Lambda[[t]]$ given by 
$$f = t^r+a_{r-1}t^{r-1}+\cdots+a_0 \ \text{where}\  a_i\in (\pi) \ \forall \  i.$$

\noindent \underline{Step 4}:
In this step, we show that~\eqref{(2)} of Step 3 holds. Let $x_1$ (resp. $x_2$) be the codimension 1 point in $X$ corresponding to the prime ideal $(t)$ (resp. $(\pi) $). Then, for $x\in X_{(1)}-\{x_1, x_2\}$, $d_x([\pi][t])=0$, since both $\pi$ and $t$ are units at $x$.
Also, since $t$ is a uniformizer of $X$ at $x_1$ and $\pi$ is a uniformizer of $x_1$ at the closed point, \begin{align*}
d'_{x_1}\circ d_{x_1}([\pi][t]) 
                                            &=-1\otimes  \-t^{\vee}\wedge  \-\pi^{\vee}.
\end{align*}

\noindent Similarly, since $\pi$ is a uniformizer of $X$ at $x_2$ and $t$ is a uniformizer of $x_2$ at the closed point, 
$$d'_{x_2}\circ d_{x_2}([\pi][t])=d'_{x_2}([t]\otimes \-\pi^{\vee})=1\otimes  \-t^{\vee}\wedge \-\pi^{\vee}.$$
Thus, $$\sum_{x\in X_{(1)}} d'_x\circ d_x([\pi][t])=-1\otimes  \-t^{\vee}\wedge  \-\pi^{\vee}+1\otimes  \-t^{\vee}\wedge \-\pi^{\vee}=0.$$

\noindent \underline{Step 5}: We now verify  \eqref{(3)} of Step 3.

Let $L={\rm Frac}(\Lambda)$ and $$h:X'=\Spec A[\pi^{-1}]\to \Spec L[t]=\A^1_L.$$ Then, $X'_{(0)}=X_{(1)}-\{x_2\}$ and
$h(X'_{(0)})$ is the set of closed points of $\A^1_L$ associated to the irreducible polynomials of the form $t^r+a_{r-1}t^{r-1}+\cdots+a_0$ where $a_i\in (\pi)$ for all $i$. For each $x\in X'_{(0)}=X_{(1)}-\{x_2\}$, $k(x)$ is a finite field extension of $L$ via $h$. We have $d'_x=\partial \circ Nm_{k(x)/L}$, where $$\partial: \KMW_{n+1}(L, ((t)/(t)^2)^{\vee})\to \KMW_{n}(\kappa, (\det \fr m/\fr m^2)^{\vee})$$ is the residue morphism associated to the discrete valuation on $L$. 
To see this, we observe from~\cite[Remark 5.20, page 127]{morel} and an analogue of~\cite[Cor. 7.4.3 page 205]{Gille-Szamuely} for Milnor-Witt K theory of fields proved in~\cite[Corollary 2.2.23]{feld20} that the following diagram 
$$\xymatrix@C=6em{
\KMW_{n+1}(\kappa(x), \Lambda_x)\ar[r]^{~\sum_{w}\partial_{w}}\ar[d]^{Nm_{\kappa(x)/L}}      & \bigoplus_{w}\KMW_{n}(L(w), \Lambda_w)\ar[d]^{\sum_{w}Nm_{L(w)/\kappa}}\\
\KMW_{n+1}(L, ((t)/(t)^2)^{\vee})\ar[r]^{\partial}      & \KMW_{n}(\kappa, \det \fr m/\fr m^2)
}$$
commutes.
Here $w$  denotes a discrete valuation of $\kappa(x)$ which restricts to the discrete valuation on $L$ and $L(w)$ is the residue field of $w$. Let $i:  \KMW_{n+1}(L(t))\to  \KMW_{n+1}(K)$ induced by the inclusion $L(t)\to K$.

Thus,
\begin{align}
\sum_{x\in X_{(1)}} d'_x\circ d_x\circ i&=d'_{x_2}\circ d_{x_2}\circ i+\sum_{x\in X_{(1)}-\{x_2\}} d'_x\circ d_x\circ i\\
&= d'_{x_2}\circ d_{x_2}\circ i+\partial \Big(\sum_{x\in X_{(1)}-\{x_2\}} Nm_{\kappa(x)/L}\circ d_x\circ i\Big).
\end{align}
Let $\alpha=[f][g_1]\cdots[g_{n+1}]\in\KMW_{n+2}(L(t)) .$
We observe that $div(f), div(g_i)$ are supported in the image $h(X'_{(0)})\subset (\A^1_L)_{(0)}$. Thus, for $x\in (\A^1_L)_{(0)}-h(X'_{(0)}), d_x(\alpha)=0.$
 
Also, at $\infty\in (\Pone_L)_{(0)}$, $f(\infty)=1$, hence $d_{\infty}\circ i(\alpha)=0$.

By the reciprocity law~\cite[Chapter 4, (4.8), page 98]{morel},
$$\sum_{x\in X_{(1)}-\{x_2\}} Nm_{\kappa(x)/L}\circ d_x\circ i(\alpha)=\sum_{x\in (\Pone_L)_{(0)}} Nm_{\kappa(x)/L}\circ d_x\circ i(\alpha)=0.$$
At $x_2$, $f(x_2)=1$, hence $d_{x_2}\circ i(\alpha)=0.$ Thus,~\eqref{(3)} holds.\\

\noindent \underline{Step 6}: It now remains to prove the claim of Step 2 in the case $m:=n+2\leq 0$. In this case $\KMW_m(K)$ is generated by $\eta^m\< u \>$, for $u\in K^{\x}$. Since the differentials $d$ commute with $\eta$, it is enough to show that $d\circ d(\<u\>)=0$ for $u\in K^{\x}$. We note that $\<u\>=1+\eta[u]$, so it suffices to show : 
\begin{align}
d\circ d([\pi])&=0 \label{pi}\\
d\circ d([t])&=0 \label{t}\\
d\circ d([h])&=0 \label{h} \ \ \ \text{where} \ h\in A^{\x} \ \text{or} \ h\in H.
\end{align}
The above are explicit and straightforward calculations along the lines of arguments used in Steps 4 and 5. We leave the details to the reader. 
\end{proof}



\noindent We now recall some definitions of morphisms between Gersten complexes. These are exactly as in \cite[Section 5]{Feld} 

\begin{definition}\label{maps}
\noindent\begin{enumerate}
\item
Let $f:X\to Y$ be a finite type morphism of equi-dimensional $S$-schemes with $\dim X=d$ and $\dim Y=e$. Let $\calL_Y$ be a graded line bundle on $Y$. Let $\calL_X$ be the pullback of $\calL_Y$ on $X$. Then there is a map 
$$f_*: C^q(X, M, \calL_X)\to C^{e-d+q}(Y, M,\calL_Y)$$ defined for 
 $\alpha\in M(x, \calL_X)$ as
 \begin{align*}({f_*})_{xy}(\alpha)=
\begin{cases}
\phi^*(\alpha) & \text{if $f(x)=y$ and $\phi: \kappa(y)\to \kappa(x)$ is a finite field extension}\\
 0 & \text{otherwise}
 \end{cases}
 \end{align*}
 where  $\phi^* :M(x, \calL_X)\to M(y, \calL_X)$.  When $\phi$ is finite morphism, $f_*$ is a morphism of complexes (Lemma \ref{cycle-prop}).
  \item 
Let $g:Y\to X$ be a smooth morphism of equi-dimensional $S$-schemes with constant relative dimension $s$. Let $\calL_X$ be a graded line bundle on $X$. Let $\calL_Y$ be the pullback of $\calL_X$ on $Y$. Then there is a map
$$g^*: C^q(X, M, \calL_X)\to C^{q}(Y, M, (\w_{Y/X}^{\vee}, -s)\ox\calL_Y)$$ defined for 
 $\alpha\in M(x, \calL_X)$ as
  \begin{align*}({g^*})_{xy}(\alpha)=
\begin{cases}
 \Theta\circ i_*(\alpha)      &\text{if $g(y)=x$ and $y$ is a generic point of $Y_x$}\\
 0    & \text{otherwise}
 \end{cases}
 \end{align*}
where  $i_*:M(x, \calL_X)=M(\kappa(x), \Lambda_x\ox\calL_X)\to M(\kappa(y), \Lambda_x\ox\calL_Y)$ and $$\Theta:  M(\kappa(y), \Lambda_x\ox\calL_Y)\to M(\kappa(y), \Lambda_y\ox(\w_{Y/X}^{\vee}, -s)\ox\calL_Y)= M(y, (\w_{Y/X}^{\vee}, -s)\ox\calL_Y)$$ is the isomorphism induced by the canonical isomorphism $$\Lambda_y\simeq \Lambda_x\ox  (\w_{Y/X}, s)$$ of graded line bundles on $\Spec \kappa(y)$.
 
 \item Let $X$ be an equi-dimensional $S$-scheme of dimension $d$. For a unit $u\in \sO(X)$, there is a map $$[a]: C^q(X, M, \calL_X)\to C^{q}(X, M, \A^1\ox\calL_X)$$ defined by 
\begin{align*}[a]_{xy}(\alpha)=
\begin{cases}
 \<-1\>^{d-q}[a]\cdot\alpha      &\text{if $x=y$ }\\
 0    & \text{otherwise}.
 \end{cases}
 \end{align*}
 
 \item Let $i: Z\hookrightarrow X$ be a closed subscheme of $X$ and $U:=X-Z\xrightarrow{j} X$ be the open complement of $Z$ in $X$. Let $\calL$ be a graded line bundle on $X$. There are term-wise maps  $j_*:C^q(U, M, \calL)\to C^q(X, M, \calL)$ and  $i^*: C^q(X, M, \calL)\to C^q_Z(X, M, \calL)$. There is a boundary map
 $$\partial^U_Z:C^q(U, M, \calL)\to C^{q+1}_Z(X, M, \calL)$$
 defined as $$\partial^U_Z:= i^*\circ d_X\circ j_*.$$
 \end{enumerate}
\end{definition}

We recall some properties satisfied by the above maps which will be used in the next section.
\begin{lemma}\label{cycle-prop}
\noindent\begin{enumerate}
\item For a finite morphism $f:X\to Y$ of $S$-schemes, $f_*\circ d_X=d_Y\circ f_*$.
\item  Let $g:Y\to X$ be a smooth morphism of equi-dimensional $S$-schemes with constant relative dimension $s$. Then $g^*\circ d_X=d_Y\circ g^*$.
\item For a unit $u\in \sO(X)$, $\gamma_{\epsilon\cdot[\-u]}\circ d_X= d_X\circ \gamma_{[u]}$. 
\end{enumerate}
\end{lemma}
\begin{proof}
\noindent\begin{enumerate}
\item The proof in \cite[Proposition 49.9]{Elm-Kar-Mer} works for $M$ as is. 
\item The proof follows as in \cite[Proposition 4.6]{Rost} using R1c and R3c.
\item This follows from R3e.
\end{enumerate}\end{proof}

\section{Local acyclicity}
In this section, we assume $S=\Spec R$, where $R$ is an excellent DVR $(R, \fr m)$ with residue field $\kappa=\kappa(s)$ and the generic point $\eta$. The goal of this section is to prove Theorem \ref{main}. This proof is an adaptation of the proofs of \cite[Theorem 6.1]{Rost} and \cite[Theorem 3.1]{lueders20}. 

\begin{lemma}
\label{boundary} Let $M$ be a MW-cycle module on $S$.
Let $g: Y\to X$ be a smooth $S$-morphism of finite type $S$-schemes of relative dimension 1. Let $\sigma:X\to Y$ be a section to g and let $t\in \mathcal{O}_Y$ be a global parameter defining the subscheme   $\sigma(X)$. Moreover, let $\tilde{g}:U=Y-\sigma(X)\to X$ be the restriction of g and let $\partial$ be the boundary map associated to the tuple $(X\xrightarrow{\sigma} Y \xleftarrow{}U)$, then $$\partial\circ\tilde{g}^*=0$$ and $$\Theta\circ\partial\circ [t] \circ \tilde{g}^*=(id_X)_*,$$
where $\Theta$ is the isomorphism from the Gersten complex for $M$ on $Y$ supported at $X$ to the Gersten complex for $M$ on $X$.
\end{lemma}
\begin{proof}
As in Rost's original proof given in \cite[Lemma 4.5]{Rost}, the proof of this theorem reduces to the case when $X$ and $S$ are fields. We refer the reader to \cite{Feld} for the proof in this case, which is a minor modification of Rost's original proof. 

We add a remark about $\Theta$: $\partial$ maps into $C^{p+1}_{X}(Y, M, \A^1\ox(\w_{Y/X}^{\vee}, -1)\ox\calL_Y)$. For $x\in X$ and $y\in U$ the generic point of the fiber $\~g^{-1}(x)$, the canonical isomorphisms $$\Lambda_y\simeq \Lambda_x\ox (\w_{Y/X}, 1)\ox \kappa(y)$$ and $$ (\w_{Y/X}^{\vee}, -1)\ox (\A^1_{\kappa(y)}, 1)=(\kappa(y) , 0)$$ of graded line bundles on $\Spec \kappa(y)$ induces the isomorphism 
$$\Theta: C^{p+1}_{X}(Y, M, \A^1\ox(\w_{Y/X}, -1)\ox\calL_Y)\simeq C^p(X, M, \calL_X ).$$\end{proof}

We recall a geometric presentation lemma from Gillet-Levine~\cite{GL87}, that will be used in the proof of Theorem \ref{main}, following \cite[6.1]{Rost}. 
\begin{lemma}(\cite[Lemma 1]{GL87})\label{gillet-levine}
Let $X$ be an affine scheme flat of finite type over a discrete valuation
ring $R$. Let $T\subset X$ be a finite set of points such that $X$ is smooth at $T$ over $R$. Let $Y$ be a principle effective divisor which is flat over $R$. Then there exists an affine open neighbourhood
$X'\subset X$ of $T$ and a morphism $\pi: X' \to \A^{d-1}_S$, where d is the relative dimension of $X$ over $S$, such that \begin{enumerate}
\item  $\pi$ is smooth of relative dimension 1 at the points of $T$. 
\item $\pi|_{Y'}: Y'\to \A^{d-1}_S$ restricted to $Y' = X'\cap Y$ is quasi-finite.
\end{enumerate}
\end{lemma}

\begin{proposition}
\label{locvanishforMW}
Let $X$ be a smooth $S$-scheme of relative dimension $d$. Let $Y\subset X$ be a closed subscheme of codim $\geq 1$ such that 
$Y$ is flat over $S$.  
Let $y$ be a point in $Y$ lying over the closed point $s\in S$. Then there is a Zariski neighbourhood $X'$ of $y$ in $X$ such that the map
$$A^p_{Y\times_X X'}(X', M, \calL_{X'})\to A^p(X', M, \calL_{X'})$$
is trivial.
\end{proposition}

\begin{proof} We follow Rost's proof of ~\cite[Proposition 6.4]{Rost}. 
By Lemma \ref{gillet-levine}, we have a Zariski neighbourhood $X'$ of $y$ and a morphism  $\pi: X'\to\A^{d-1}_S$ of relative dimension 1 which is smooth at $y$ such that  $\pi|_{Y'}: Y'\to \A^{d-1}_S$ is quasi-finite where $Y':= Y\times_XX'$ . Let $Z:=Y\times _{ \A^{d-1}_S} X'$ and $\sigma: Y'\to Z$ be the section of $Z\to  Y'$ induced by the morphism $Y'\to X'$. We have the commutative diagram
\begin{equation}\label{gabberpres}\xymatrix{
Y' \ar[rd]^-{\sigma}\ar@/_1pc/[rdd]_{~1_Y}\ar@/^1pc/[rrd]^{i}&        & & \\
   & Z    \ar[r]^{~p}\ar[d]^{q}  & X' \ar[d]_{\pi}\\
   & Y'    \ar[r]^{\pi|_{Y'}}  & \A^{d-1}_S\\
}\end{equation}

By Zariski's main theorem, we factorize the quasi-finite map $Z\xrightarrow{p} X'$ as an open immersion $Z\to \-Z$ followed by a finite map $\-p: \-Z\to X'$, so that we have the following commutative diagram 
\begin{equation}\label{gabberpres}\xymatrix{
Y' \ar[rd]^{\sigma}\ar[rdd]_{~1_Y}\ar[rr]^{\-\sigma}&        &\-Z \ar[d]^{\-p}& D:=\-Z-Z\ar[l]_-{r}\\
   & Z    \ar[r]^{~p}\ar[d]^{q}\ar[ru]^{}  & X' \ar[d]_{\pi}\\
   & Y'    \ar[r]^{\pi|_{Y'}}  & \A^{d-1}_S\\
}\end{equation}
By passing to a smaller open subscheme of $X'$, we can assume that $\sigma(Y')\subset Z$ is given by a section $t\in \sO_{Z}(Z)$ and $D\cap \-\sigma(Y')=\emptyset$. As in the proof of~\cite[Theorem 4.1]{Geis04}, by Chinese remainder theorem, by further shrinking to a smaller open subscheme, we can find $\-t\in\sO(\-Z)$ such that $\-t|_{Z}=t$ and $\-t|_D=1$. 

Let $Q:=Z-\sigma(Y')$, $\-Q:=\-Z-\-\sigma(Y')$, $j: Q\to \-Q$ and $\-j: \-Q\to \-Z$. Consider the correspondence 
$$\xymatrix{ 
H: Y'\ar@{{*}->}[r]^{q^*} & Q\ar@{{*}->}[r]^{[\-t]} & Q\ar@{{*}->}[r]^{j_*} & \-Q\ar@{{*}->}[r]^{\-j_*} & \-Z\ar@{{*}->}[r]^{\-p_*}& X' \\
}$$
giving a homomorphism 
$$H: C^p_{Y'}(X', M, \calL_{X'})\to C^p(X', M, \calL_{X'}).$$
We claim that $H$ satisfies the following relation
\begin{equation}
\label{homMW}
d_{X'}\circ H- \epsilon H\circ d_{Y'}=i_*
\end{equation}
where $i:Y'\inj X'$ is the closed immersion and $\epsilon=-\<-1\>\in \KMW_0(S)$. 
\begin{align*}
d_{X'}\circ H&=d_{X'}\circ \-p_*\circ\-j_*\circ j_*\circ [\-t]\circ q^*\\
                 &\overset{(1)}= \-p_*\circ d_{\-Z}\circ\-j_*\circ  j_*\circ [\-t] \circ q^*  \\
                 &\overset{(2)}= \-p_*\circ(\sigma_*\circ \Theta\circ\partial^Q_{Y'}+{i_{\-Z-Z}}_*\circ \partial^{Q}_{\-Z-Z}+\-j_*\circ j_*\circ d_Q)\circ [\-t] \circ q^*\\
                 &\overset{(3)}= \-p_*\circ \sigma_*\circ\Theta\circ \partial^Q_{Y'}\circ [\-t] \circ q^*+ \-p_*\circ j_*\circ d_Q\circ[\-t] \circ q^*\\
                 &\overset{(4)}=\-p_*\circ \sigma_*\circ {1_{Y'}}_*+ \-p_*\circ \-j_*\circ j_*\circ (\epsilon[\-t]\circ d_Q) \circ q^*\\
                 &\overset{(5)}=i_*+\-p_*\circ \-j_*\circ j_*\circ(\epsilon[\-t])\circ q^*\circ d_{Y'}\\
                 &\overset{(6)}=i_*+ \epsilon \-p_*\circ \-j_*\circ j_*\circ [\-t]\circ q^*\circ d_{Y'}\\
                 &=i_*+ \epsilon H\circ d_{Y'}. 
\end{align*}
(1) follows since $\-p$ is finite, hence $d_{X'}\circ \-p_*=\-p_*\circ d_{\-Z}$ [see Lemma~\ref{cycle-prop}(1)].\\
(2) follows since $\sigma_*\circ\sigma^*+{i_{\-Z-Z}}_*\circ {i_{\-Z-Z}}^*+ \-j_*\circ \-j^*=1_{\-Z}$, by the argument similar to the one in \cite[(3.10), page 350]{Rost}. Thus  
$$\sigma_*\circ\sigma^*\circ d_{\-Z}\circ \-j_*\circ j_*+{i_{\-Z-Z}}_*\circ {i_{\-Z-Z}}^*\circ d_{\-Z}\circ \-j_*\circ j_*+ \-j_*\circ \-j^*\circ d_{\-Z}\circ\-j_*\circ j_*=d_{\-Z}\circ\-j_*\circ j_*.$$ Since $\Theta\circ\partial^Q_Y=\sigma^*\circ d_{\-Z}\circ  \-j_*\circ j_*$ (recall from Lemma \ref{boundary} the definition of $\Theta$), we have
$$d_Z\circ \-j_*\circ j_*=\sigma_*\circ \Theta\circ\partial^Q_{Y'}+{i_{\-Z-Z}}_*\circ \partial^{Q}_{\-Z-Z}+\-j_*\circ j_*\circ d_Q.$$\\
(3) follows since $\-t|_{\-Z-Z}=1$ implies $ \partial^{Q}_{\-Z-Z}\circ [\-t] =0$. \\
(4) First term in (4) follows by  Lemma~\ref{boundary} applied to triple $(Y'\xrightarrow{\sigma}Z\xleftarrow{j}Q).$ The second term follows since $d_Q\circ [t]=\epsilon[t]\circ d_Q$ \cite[Prop. 6.6 (3) page 22]{Feld}. \\ 
(5) follows since $\-p\circ \sigma=i$ and because $d_Q\circ q^* = q^*\circ d_Y'$ (see Lemma \ref{cycle-prop}(2)). \\ 
(6) follows since $\epsilon$ commutes with $p_*$ by projection formula~R2b
and also commutes with $j_*$ and $\-j_*$.\\
The relation~\eqref{homMW} implies that $i_*: A^p_{Y\times_X X'}(X', M, \calL_{X'})\to A^p(X', M, \calL_{X'})$ is zero.
\end{proof}

We are now ready to prove Theorem \ref{main}. 
\begin{proof}[Proof of Theorem \ref{main}] We need to show that $$ A^p(X, M, \calL_{X})=0 \ 
\forall \ p\geq 1.$$The proof is an adaptation of the ideas in~\cite[Proposition 6.1]{Rost} and~\cite[Theorem 3.1]{lueders20}) with input from Proposition~\ref{locvanishforMW}.
For simplicity, we prove the theorem when $X$ is local. The proof of the semi-local case is very similar with just more notation. 

Let $X=U_u:=\Spec \mathcal{O}_{U, u}$ where $U$ is a smooth, irreducible, equi-dimensional $S$-scheme of dimension $d$ and a point $u$ in $U$ lying over the closed point $s$ in $S$. We have
$$C^p(X, M, \calL_X)=\dlim_{(U', u)} C^p(U', M, \calL_{U'})$$ where $(U', u)$ runs over Zariski neighbourhoods $U'$ of $u$ in $U$ and
$$C^p(U', M, \calL_{U'})=\dlim_{Y\subset U'} C^{p}_Y(U', M, \calL_{U'})$$
where $Y$ runs over closed subschemes of $U'$ of dimension $d-p$.
 
We have $$A^p(X, M, \calL_{X})=    \dlim_{(U', u)} A^p(U', M, \calL_{U'})    $$
where $(U', u)$ runs over Zariski neighbourhoods of $u$ in $U$.
Thus, every element of  $A^p(U', \KMW_n)$ is represented by an element in $C^{p}_Y(U', M, \calL_{U'})$
for some closed subscheme $Y$ of $U'$ of dimension $d-p$. 
 


We claim that, every element of $A^{p}(U', M, \calL_{U'})$ is represented by an element of $C^{p}_Y(U', M, \calL_{U'})$ where $Y$ runs over all closed subschemes of $ U'$ of dimension $d-p$ and $Y$ is flat over $S$. 

Let $y\in Y^{(0)}$ such that $y\in U'_s$, the special fiber of $U'$ over the closed point $s$ of $S$. By  ~\cite[Lemma 7.2, page 1622]{SS}
there exists an integral closed subscheme $Z\subset U'$ of dimension $d-p+1$  such that 
\begin{enumerate}
\item $ Z\cap (U'-U_s)\neq \emptyset$
\item $y\in Z$
\item $Z$ is regular at $y$.
\end{enumerate}
Let $\alpha\in M(\kappa(y), \calL_{U'}).$ The set of all codimension 1 points of $Z$ is given by $Z^{(1)}=Z_s^{(1)}\coprod Z_{\eta}^{(1)}$ where  $Z_s^{(1)}:=\bigg\{\-{\{y\}}\bigg\}\cup  \bigg\{Z_1, Z_2,\cdots Z_m\bigg\}$ where $Z_i$'s are pairwise distinct irreducible components of the closed fiber $Z_s$ of $Z$ over the closed point $s$ of $S$ and not equal to $\-{\{y\}}$ and $Z_{\eta}^{(1)}$ is the collection of codimension 1 points of $Z$ not in $Z_{s}^{(1)}.$  

Let $z$ be the generic point of $Z$. Let $z_i$ be the generic points of $Z_i$ for $i=1,\cdots, m$. Let $J_i$ be the collection of discrete valuations on $\kappa(z)$ with center $z_i$ for all $i$. Let us denote the valuation on $\kappa(z)$ associated to $y$ by $v$. Observe that $\{v\}\cup \bigcup_i J_i$ is a finite collection of distinct non-equivalent discrete valuations on $\kappa(z)$. 

By Chinese remainder theorem, we can choose a uniformizer $\pi$ for $v$ such that $\pi+1\in \fr m_w$ for all $w\in \bigcup_i J_i$, where $\fr m_w$ is the maximal ideal for a discrete valuation $w$. Now, since $M$ satisfies the property R5, there is $\beta\in M(z, \sL_{U'})$ such that $\partial_v(\beta)=\alpha$  and for all $w\in \bigcup_i J_i$, $\partial_w(\beta)=0$. Hence, $\forall \ i,$ the map
$$\partial^z_{z_i}=\sum_{w\in J_i} (\phi^*)_{\kappa(w)/\kappa(z_i)}\circ \partial_{w}: M(z, \calL_{U'})\to  M(z_i, \calL_{U'})$$ takes $\beta\mapsto 0.$
Then under the differential
$$d=\sum_{w\in Z^{(1)}}\partial^z_w: M(z, \calL_{U'})\to \Bigg(\bigoplus_{i=1}^m M(z_i, \calL_{U'})\Bigg) \oplus M(y, \calL_{U'})\oplus \Bigg(\bigoplus_{w\in Z_{\eta}^{(1)}} M(w, \calL_{U'})\Bigg)$$
$$d(\beta)=0+ \cdots+ 0+ \alpha+ \sum_{w\in Z_{\eta}^{(1)}} \partial^{z}_w(\beta).$$
We observe that $\sum_{w\in Z_{\eta}^{(1)}} \partial^{z}_w(\beta)\in C^p(U', M, \calL_{U'})$ is supported on $w$, where $\-{\{w\}}$ is a dimension $d-p$ closed subscheme and flat over $S$, since $w\in Z_{\eta}^{(1)}$. 
Thus, we have that, any element of $A^{p}(U', M, \calL_{U'})$ is represented by 
an element in $A^{p}_Y(U', M, \calL_{U'})$
for some closed subscheme $Y\subset U'$ of codimension $p$ such that $Y$ is flat over $S$.
Now  $$A^{p}_Y(U',M, \calL_{U'})\to A^{p}(U', M, \calL_{U'}) \to A^{p}(U'', M, \calL_{U''})$$ is trivial for some $(U'', u)$ Zariski neighbourhood of $u$ in $U'$, since it factors through $$A^{p}_{Y}(U', M, \calL_{U'})\to A^{p}_{Y\times_{U'} U''}(U'', M, \calL_{U''})\to A^{p}(U'', M, \calL_{U''})$$
which is trivial by Proposition~\ref{locvanishforMW}.
\end{proof}

\section{Strict $\A^1$-invariance}
Throughout this section, we let $S$ be as in Theorem \ref{main}. The goal of this section is to give a slightly modified construction of a spectral sequence constructed in \cite[9.2]{Feld}. We show that this spectral sequence can easily be constructed using a filtration on the Gersten complex which makes it clear that the perfectness of the field is not a necessary hypothesis as might appear while reading the proof of \cite[9.3]{Feld}. In fact, the spectral sequence exists over any regular excellent base scheme $S$. As in \cite[9.1]{Feld}, this spectral sequence will deduce the proof of Theorem \ref{a1inv} from the field case. 

We first recall the definition of the sheaf theoretic version of the Gersten complex in Definition~\ref{cyclecomplex}. For a smooth scheme $f: X\to S$ and a graded line bundle $\sL$ on $X$, we let $$\calC^q(X, M, \calL)=\bigoplus_{x\in X^{(q)}} {i_x}_*M(x, \calL).$$ 
where $i_x:\Spec(\kappa(x))\to X$ is the inclusion of the point $x\in X$ and the group $M(x,\sL_X)$ is considered as a Nisnevich sheaf on $\Spec(\kappa(x))$. Note that the Gersten complex $C^{\bullet}(X,M,\sL)$ in Definition~\ref{cyclecomplex} is obtained by applying global sections functor to $\calC^{\bullet}(X, M,\calL)$. Moreover,~Theorem~\ref{main} implies that $\calC^{\bullet}(X, M, \sL)$ is an acyclic resolution of the sheaf $M(X,\sL)$ in the category of Nisnevich sheaves on $X$.  In particular, we have 
$$H^q_{\nis}(X, M(X,\sL)) \cong A^q(X, M, \sL) \ \ \forall\  q\geq 0.$$
One now defines a decreasing filtration on the complex $\sC^{\bullet}(X, M, \sL)$ by 
setting 
$$F^p\sC^{q}(X, M, \sL)=\bigoplus_{\substack{x\in X^{(q)}\\ f(x)\in S^{(\geq p)}}} M(x, \sL).$$
\begin{lemma}\label{spectralseq}
We have an isomorphism of complexes
$$F^p\sC^{\bullet}(X,M,\sL)/F^{p+1}\sC^{\bullet}(X,M,\sL) =\bigoplus_{s\in S^{(p)}}C^{\Dot-p}(X_s, M, \calL_{X_s}).$$
Hence, we obtain the spectral sequence
$$E_1^{pq}=\bigoplus_{s\in S^{(p)}}A^q(X_s, M, \calL_{X_s})\Longrightarrow A^{p+q}(X, M, \calL).$$
\end{lemma}
\begin{proof} Let $X\xrightarrow{f}S$ denote the structure map. It is elementary to check that as groups we have an isomorphism of groups
$$F^p\sC^{\bullet}(X,M,\sL)/F^{p+1}\sC^{\bullet}(X,M,\sL) =\bigoplus_{s\in S^{(p)}}C^{q-p}(X_s, M, \calL_{X_s})$$
since both are isomorphic to 
$$\bigoplus_{s\in S^{(p)}}\Bigg(\bigoplus_{x\in X_s^{(q-p)}} M(x, \calL_X)\Bigg)$$

To see that, this is indeed an isomorphism of complexes, we only need to check that there are no non-trivial boundary maps from $M(x,\sL)$ to $M(y,\sL)$ in $F^p(C^{\Dot})/F^{p+1}(C^{\Dot})$ unless $x$ and $y$ belong to the same fiber, i.e. we need to show that if $x$ specializes to $y$ then either 
$f(x)=f(y)$ or $f(y)\in S^{(\geq p+1)}$. This is clear, because $f(x)$ speciallizes to $f(y)$ and both cannot have the same codimension unless $f(x)=f(y)$. 
\end{proof}

\begin{proof}[Proof of Theorem \ref{a1inv}]
This is deduced from the spectral sequence \eqref{spectralseq} and the Proposition \ref{hwr}(H) as in \cite[Theorem 9.4]{Feld}.
\end{proof}

\section{Appendix}
In this appendix, we quickly recall facts of canonical sheaves \cite{liu} and fix some notation required in this paper. Nothing in this appendix is new. 

\subsection{Canonical invertible sheaf}\label{canonicalsheaf}
Let $S$ be a J-2 ~\cite[\href{https://stacks.math.columbia.edu/tag/07P7}{Tag 07P7}]{stacksproj}, Noetherian regular scheme of finite Krull dimension, in particular $S$ could be an excellent scheme as in the previous sections. In \cite[Chapter 6, Definition 4.3]{liu} canonical sheaf is defined for all l.c.i morphisms. For a point $x\in X$, where $X/S$ is of finite type, we wish to define $\omega_{x/S}$. 

\begin{definition}\label{canonicalsheaf}
Consider $\-{\{x\}}$ the closure of $x$ in $X$ with the reduced induced structure. Now $\-{\{x\}}$ being a reduced, finite type scheme over $S$, the regular locus $U_x:={\rm Reg}(\-{\{x\}})$ is open and non-empty. The morphism $U_x\to S$ is l.c.i. over $S$, since it is of locally finite type morphism of regular schemes. There is a canonical invertible sheaf $\w_{{U_x}/S}$ on $U_x$ as  defined in ~\cite[Chapter 6, Definition 4.3]{liu}. We define $$\w_{x/S}:= (\w_{{U_x}/S})_x\otimes_{\sO_x} \kappa(x).$$
We leave it to the reader to check that this  $\w_{x/S}$ depends only on the $S$-field $\kappa(x)$.
Thus, for any finitely generated $S$-field $F$, we have a well defined canonical sheaf $\omega_{F/S}$. 
\end{definition}

\begin{remark}\label{rmkcanonicalsheaf}
We note various properties of the canonical sheaf.
\begin{enumerate}
\item\label{rmklci} If the given morphism $X\to S$ is l.c.i., then for any regular point $x\in X$  we have $$\w_{x/S}\simeq (\det\frak m_x/\frak m_x^2)^{\vee}\otimes_{\kappa(x)} ((\w_{X/S})_x\otimes_{\sO_{X, x}} \kappa(x)).$$
\item When $S=\Spec(k)$ is a field, then for any finitely generated field extension $F/k$, $\omega_{F/k} \cong \det \Omega_{F/k}$.
\item \label{basechange}Given field extensions $F\supset K$ of finitely generated $S$-fields, there is a canonical isomorphism $$\omega_{F/S} \simeq \omega_{F/K} \tensor \omega_{K/S}.$$
\end{enumerate}
\end{remark}

\subsection{Graded line bundles} 
We recall the notation from~\cite[section 1.3]{Fas20}. Let $X$ be a connected scheme. The category $\calG(X)$ of graded line bundles on $X$ is defined as follows. The objects are a pair $\sL:=(L, a)$ for a line bundle $L$ on $X$ and an integer $a$. The morphisms are given as follows
\begin{align}
\Hom_{\calG(X)}((L, a), (L', a'))=\begin{cases}
        \emptyset & \text{if $a\neq a'$}\\
        \Hom_{\sO_X}(L, L') &\text{if $a= a'$}
\end{cases}
\end{align}
$\calG(X)$ is a symmetric monoidal category with monoidal structure given as follows:
$$(L, a)\otimes (L', a')=(L\otimes L', a+a').$$ 
$\calG(X)$ is a Picard category in the sense of~\cite[\S 4]{Del87}.

\end{document}